\documentclass[a4paper,11 pt]{amsart}

\RequirePackage{amsmath, amssymb, amsthm, amsfonts}
\usepackage{fullpage}
\newtheorem{theo}{Theorem}[section]
\newtheorem{lem}[theo]{Lemma}
\newtheorem{prop}[theo]{Proposition}

\newtheorem{defin}[theo]{Definition}

\theoremstyle{definition}

\newtheorem{rem}[theo]{Remark}

\usepackage[all]{xy}
\usepackage{color}

\newcommand{\void}[1]{}     

\newcommand{\PP}{\mathbb{P}}

\newcommand{\ord}{{\mathrm{ord}}}

\newcommand{\ZZ}{\mathbb{Z}}

\newcommand{\sym}{\mathbf{S}}

\title{On Zariski Multiplets of Branch Curves from Surfaces Isogenous to a Product}
\author{Michael L\"onne and Matteo Penegini}
\address{Michael L\"onne\\  %Lehrstuhl &new Adresse verbessert
Mathematik VIII, Universit\"at Bayreuth,
NWII, Universit\"atsstrasse 30, D-95447 Bayreuth,  Germany}
\email{michael.loenne@uni-bayreuth.de}
\address{Matteo Penegini\\
Universit\`a degli Studi di Genova\\
Dipartimento di Matematica - DIMA\\
Via Dodecaneso 35,
I-16146 Genova (GE) - 
Italy} 
\email{penegini@dima.unige.it}

\begin{document}

%%%%%%%%%%%%%%%%%%%%%%%%%%%%%%%%%%%%%%%%%%%%%%%%%%%%%%%%%%%%%%%%%%%%%%%%%%%%%%%%%%%%%%%%%%%
%%%%%%%%%%%%%%%%%%%%%%%%%%%%%%%%%%%%%%%%%%%%%%%%%%%%%%%%%%%%%%%%%%%%%%%%%%%%%%%%%%%%%%%%%%%

\maketitle

%%%%%%%%%%%%%%%%%%%%%%%%%%%%%%%%%%%%%%%%%%%%%%%%%%%%%%%%%%%%%%%%%%%%%%%%%%%%%%%%%%%%%%%%%%%
%%%%%%%%%%%%%%%%%%%%%%%%%%%%%%%%%%%%%%%%%%%%%%%%%%%%%%%%%%%%%%%%%%%%%%%%%%%%%%%%%%%%%%%%%%
%%%%%%%%%%%%%%%%%%%%%%%%%%%%%%%%%%%%%%%%%%%%%%%%%%%%%%%%%%%%%%%%%%%%%%%%%%%%%%%%%%%%%%%%%%

\begin{abstract}
In this paper we give an asymptotic bound of the cardinality of Zariski multiples of particular plane singular curves. These curves have only nodes and cusps as singularities and are obtained as branched curves of ramified covering of the plane by 
surfaces isogenous to a product of curves with group $(\ZZ/2\ZZ)^k$. The knowledge of the moduli space of these surfaces will enable us to produce Zariski multiplets whose number grows subexponentialy in function of their degree.
\end{abstract}

%%%%%%%%%%%%%%%%%%%%%%%%%%%%%%%%%%%%%%%%%%%%%%%%%%

\section{Introduction}

The notion of \emph{Zariski pair} was introduced by E. Artal in his seminal paper 
\cite{artal94}.
This put a new focus on an important question for plane complex projective curves:
\begin{quote}
Do the topological types of the singularities of a plane curve $C$ and the degree
determine the homeomorphism type of the pair $\PP^2,C$?
\end{quote}
Obviously this question has positive answer in case of plane curves without 
singularities, but already the case of nodal curves took a long time to be
established beyond doubt, cf.\ Severi \cite{S21}, Harris \cite{H86}.
In these, and many other cases the proof shows the connectedness of the
locus of curves with given singularity types in the projective space of plane
curves of the given degree. 

Instead, Zariski proved that in degree $6$ the complement of a sextic, regular
except for six cusps, has non-abelian fundamental group only if the cusps
lie on a conic, so to honour his results \emph{Zariski pairs} were named.

The excellent survey \cite{ACT10} can report on many new examples
and provides some help to find additional ones in two steps
\begin{enumerate}
\item[(I)]
Locate curves of the same invariants in different connected components
of the corresponding locus,
\item[(II)]
Find an effective invariant of the embedded pair, which distinguishes the
curves.
\end{enumerate}
We will use the fundamental group of the complement together with the
conjugacy class of meridians to the curve, which is well-defined for
irreducible curves, in the second step.

For the first step our idea is to exploit the disconnectedness of moduli
of surfaces of general type isogenous to a product.
As these surface have ample canonical class, the $m$-canonical map
is an embedding into a projective space $\PP^N$
at least for $m$ sufficiently large.

An embedded surface then is mapped onto $\PP^2$ using projection
with center a disjoint projective subspace of codimension $3$.
By the theorem of Ciliberto and Flamini \cite{CiF11} the corresponding branch
curve is reduced, irreducible and smooth except for ordinary nodes
and cusps if the center is sufficiently general.

This procedure extends nicely to families of 
$m$-canonical embeddings and
centers and thus associates a connected component of equisingular
plane irreducible curves to a connected component of surfaces of
general type with ample canonical bundle.

Our main theorem on Zariski multiplets of high cardinality follows
our results \cite{GP14,LP14,LP16} on high cardinalities of the set of connected components
of moduli spaces of surfaces isogenous to a product (SIP) and exploits the 
solution of the Chisini conjecture by Kulikov \cite{K08} for generic branched
projections.

\begin{theo}\label{thm main}
There is a $N_d$-multiplet of irreducible plane curves of degree $d$
with 
\[
n_d= \frac12d^2 - \frac{233}{28}d \mbox{ nodes},\quad
c_d= \frac{135}{28} d \mbox{ cusps}
\]
for any $\varepsilon>0$ and infinitely many $d$ such that
\[
\log_2 N_d \quad \geq\quad (\mbox{$\frac8{14}$}d)^{\frac{1}{2+\varepsilon}}. 
%{\scriptstyle{\frac8{14}}}
\]
\end{theo}

Previously Artal and Tokunaga \cite{AT04} had shown that the cardinality is unbounded but
only provided examples where the cardinality grows linearly with the
degree.

Degtyarev \cite{D09} instead has given sets of cardinality growing exponentially with
the degree of irreducible plane curves which are pairwise not equisingular deformations.
However it is not know, whether they are topological distinct as well. Apart from ordinary nodes
and cusps his examples have a singular point of multiplicity $d-3$.
\medskip

Let us explain now the way in which this paper is organized.

In the next section \emph{Preliminaries on surfaces isogenous to a product}  we recall the definition and the properties of surfaces isogenous to a higher product and its associated group theoretical data. Moreover, we recall the weak rigidity theorem of Catanese \cite{cat00} which allows us to describe the connected components of the moduli space of surfaces isogenous to a product.

In the third section we give the definition of Zariski multiplets and generic m-canonical branch curves. 
These last ones are branch curves of a covering $p\colon S \longrightarrow \PP^2$, with $S$ of general type. 
Moreover, the covering $p$ is obtained as the composition of the $m$-canonical embedding of $S$ 
and a projection onto $\PP^2$ by a generic disjoint center. 
Then, we shall see that the 
associated generic m-canonical branch curves form a Zariski pair (see Theorem \ref{theo_ZP})
if we start from a pair of  
surfaces are isogenous to a product with the same Euler number but different fundamental group.

In the last section using the results of \cite{LP16} we give the proof of the Theorem \ref{thm main}.
\medskip

\textbf{Acknowledgement} The first author was  supported by  the ERC 2013 Advanced Research Grant - 340258 - TADMICAMT,
at the Universit\"at Bayreuth. The second author was partially supported by MIUR PRIN 2015``Geometry of Algebraic Varieties'' and also by GNSAGA of INdAM. 
Both authors thank Fabrizio Catanese for the invitation to Cetraro where the paper was completed.

%%%%%%%%%%%%%%%%%%%%%%%%%%%%%%%%%%%%%%%%%%%%%%%%%%
\section{Preliminaries on surfaces isogenous to a product}

Surfaces isogenous to a product were introduced by Catanese in \cite{cat00}. Thanks to their simple definition they are incredibly versatile and they give rise to a large amount of interesting examples.  

\begin{defin}\label{def.isogenous} A surface $S$ is said to be \emph{isogenous to a higher product of curves}\index{Surface isogenous to a higher product of curves} if and only if, $S$ is a
quotient $(C_1 \times C_2)/G$, where $C_1$ and $C_2$ are curves of
genus at least two, and $G$ is a finite group acting freely on
$C_1 \times C_2$. (To ease notation we call the surfaces {\bf SIP})
\end{defin}
Using the same notation as in Definition \ref{def.isogenous}, let
$S$ be a SIP, and $G^{\circ}:=G
\cap(Aut(C_1) \times Aut(C_2))$. Then $G^{\circ}$ acts on the two
factors $C_1$, $C_2$ and diagonally on the product $C_1 \times
C_2$. If $G^{\circ}$ acts faithfully on both curves, we say that
$S= (C_1 \times C_2)/G$ is a \emph{minimal
realization}. In \cite{cat00}
it is also proven that any
SIP admits a unique minimal realization. 

\medskip

{\bf Assumptions.} In the following we always assume:
\begin{enumerate}
\item Any SIP  $S$  is given by its unique minimal realization;
\item $G^{\circ}=G$, this case is also known as \emph{unmixed type}, see \cite{cat00}.
\end{enumerate}
Under these assumption we have. \nopagebreak
\begin{prop}~\cite{cat00}\label{isoinv}
Let $S=(C_1 \times C_2)/G$ be a SIP, then $S$ is a minimal surface of general type with the following invariants:
\begin{equation}\label{eq.chi.isot.fib}
\chi(S)=\frac{(g(C_1)-1)(g(C_2)-1)}{|G|},
\quad
e(S)=4 \chi(S),
\quad
K^2_S=8 \chi(S).
\end{equation}
The irregularity of these surfaces is computed by
\begin{equation}\label{eq_irregIsoToProd}
q(S)=g(C_1/G)+g(C_2/G).
\end{equation}
Moreover the fundamental group $\pi_1(S)$ fits in the following short exact sequence of groups
\begin{equation}\label{eq_fundGroupS}
1 \longrightarrow \pi_1(C_1) \times \pi_1(C_2) \longrightarrow \pi_1(S) \longrightarrow G \longrightarrow 1.
\end{equation}
\end{prop}
Among the nice features of SIPs, one is that their deformation class can be obtained in a purely algebraic way. Let us briefly recall this in the particular case when $S$ is regular, i.e., $q(S)=0$, $C_i/G \cong \PP^1$. % hence by Equation \eqref{eq_irregIsoToProd}.  
\begin{defin}\label{defn.sphergen}
Let $G$ be a finite group and $r \in \mathbb{N}$ with $r \geq 2$.
\begin{itemize}
\item An $r-$tuple $T=(v_1,\ldots,v_r)$ of elements of $G$ is called a
\emph{spherical system of generators} of $G$ if $ \langle
v_1,\ldots,v_r \rangle=G$ and $v_1 \cdot \ldots \cdot v_r=1$.

\item We say that $T$ has an \emph{unordered type} $\tau:=(m_1, \dots ,m_r)$ if the
orders of $(v_1,\dots,v_r)$ are $(m_1, \dots ,m_r)$ up to a
permutation, namely, if there is a permutation $\pi \in \mathfrak{S}_r$ such
that
\[
   \ord(v_1) = m_{\pi(1)},\dots,\ord(v_r)= m_{\pi(r)}.
\]

\item Moreover, two spherical systems $T_1=(v_{1,1}, \dots , v_{1,r_1})$ and $T_2=(v_{2,1}, \dots , {v}_{2,r_2})$
are said to be \emph{disjoint}, if:
\begin{equation}\label{eq.sigmasetcond} \Sigma(T_1)
\bigcap \Sigma(T_2)= \{ 1 \},
\end{equation}
where
\[ \Sigma(T_i):= \bigcup_{g \in G} \bigcup^{\infty}_{j=0} \bigcup^{r_i}_{k=1} g \cdot v^j_{i,k} \cdot
g^{-1}.
\]
\end{itemize}
\end{defin}
We shall also use the shorthand, for example $(2^4,3^2)$, to indicate
the tuple $(2,2,2,2,3,3)$. 
\begin{defin}\label{def.rami.structure} Let $2<r_i \in \mathbb{N}$ for $i=1,2$ and $\tau_i=(m_{i,1}, \dots ,m_{i,r_i})$ be two sequences of natural numbers such
that $m_{k,i} \geq 2$.  A \emph{(spherical-) ramification structure} of type $(\tau_1,\tau_2)$ and size
$(r_1,r_2)$ for a finite group $G$, is a
pair $(T_1,T_2)$
of disjoint spherical systems of generators of $G$, whose types are
$\tau_i$, such that:
\begin{equation}\label{eq.Rim.Hur.Condition}
\mathbb{Z} \ni \frac{|G|
(-2+\sum^{r_i}_{l=1}(1-\frac{1}{m_{i,l}}))}{2}+1 \geq 2,\qquad \text{for } i=1,2.
\end{equation}
\end{defin}
\begin{rem}\label{minimal}
Following e.g., the discussion in \cite[Section 2]{LP14} we obtain that the datum of the deformation class of a regular SIP
$S$ of unmixed type together with its minimal realization $S=(C_1 \times
C_2)/G$ is determined by the datum of a finite
group $G$ together with two
disjoint spherical systems of generators $T_1$ and $T_2$
up to simultaneous action of $\operatorname{Aut}(G)$ and separate Hurwitz equivalence
(for more details see also \cite{BCG06,P13}).
\end{rem}
\begin{rem} Recall that from Riemann Existence Theorem a finite group $G$ acts as a
group of automorphisms of some curve $C$ of
genus $g$ such that $C/G \cong \PP^1$ if and only if there exist integers $m_r
\geq m_{r-1} \geq \dots \geq m_1\geq 2$ such that $G$ has a spherical system of generators of type $(
m_1,\dots,m_r)$ and the following
Riemann-Hurwitz relation holds:
\begin{equation}\label{eq.RiemHurw} 2g-2=| G | (-2 +
\sum^r_{i=1}(1-\frac{1}{m_i})).
\end{equation}
\end{rem}
\begin{rem} Note that a group $G$ and a ramification structure determine the main numerical
invariants of the surface $S$. Indeed, by \eqref{eq.chi.isot.fib} and~\eqref{eq.RiemHurw} we obtain:
%\begin{equation}\label{eq.pginfty}
%4\chi(S)=|G|\cdot\left({-2+\sum^{r_1}_{k=1}(1-\frac{1}{m_{1,k}})}\right)
%\cdot\left({-2+\sum^{r_2}_{k=1}(1-\frac{1}{m_{2,k}})}\right)=: 4\chi (|G|,(\tau_1,\tau_2)).
%\end{equation}
\begin{eqnarray}\label{eq.pginfty}
4\chi(S)&=\,\,&|G|\cdot\left({-2+\sum^{r_1}_{k=1}(1-\frac{1}{m_{1,k}})}\right)
\cdot\left({-2+\sum^{r_2}_{k=1}(1-\frac{1}{m_{2,k}})}\right)
\notag\\
&=:& 4\chi (|G|,(\tau_1,\tau_2)).
\end{eqnarray}
\end{rem}
Let $S$ be a SIP of unmixed type with group $G$ and a pair of
two disjoint spherical systems of generators of types
$(\tau_1,\tau_2)$. By~$\eqref{eq.pginfty}$ we have
$\chi(S)=\chi(G,(\tau_1,\tau_2))$, and hence,
by~\eqref{eq.chi.isot.fib}, $K^2_S=K^2(G,(\tau_1,\tau_2))=8\chi(S)$.

The most important property of surfaces isogenous to a product is their weak rigidity property.
\begin{theo}~\cite[Theorem 3.3, Weak Rigidity Theorem]{cat04}
\label{weak}
Let $S=(C_1 \times C_2)/G$ be a surface isogenous to a higher
product of curves. Then every surface with the same
\begin{itemize}
\item topological Euler number and
\item fundamental group
\end{itemize}
is diffeomorphic to $S$. The corresponding  moduli space
$\mathcal{M}^{top}(S) = \mathcal{M}^{\it diff}(S)$ of surfaces
(orientedly) homeomorphic (resp. diffeomorphic) to $S$ is either
irreducible and connected or consists of two irreducible connected
components exchanged by complex conjugation.
\end{theo}

\begin{rem}
Thanks to the Weak Rigidity Theorem, we have  that the moduli space
of surfaces isogenous to a product of curves with fixed invariants
--- a finite group $G$ and a type $(\tau_1,\tau_2)$ --- consists of a finite number of irreducible connected
components of $\mathcal{M}$. More precisely, let $S$ be a surface
isogenous to a product of curves of unmixed type with group $G$
and a pair of disjoint systems of generators of type
$(\tau_1,\tau_2)$. By~$\eqref{eq.pginfty}$ we have
$\chi(S)=\chi(|G|,(\tau_1,\tau_2))$, and consequently,
by~\eqref{eq.chi.isot.fib}
$K^2_S=K^2(|G|,(\tau_1,\tau_2))=8\chi(S)$, and
$e(S)=e(|G|,(\tau_1,\tau_2))=4\chi(S)$. Moreover, recall that the fundamental group of $S$ fits into the exact sequence \eqref{eq_fundGroupS} and
the subgroup $\pi_1(C_1) \times \pi_1(C_2)$  of $\pi_1(S)$ is unique, see \cite{cat00}. 
\end{rem}

%%%%%%%%%%%%%%%%%%%%%%%%%%%%%%%%%%%%%%%%%%%%%%%%%%
\section{Surfaces isogenous to a product and Zariski Pairs}

\begin{defin}
A \emph{Zariski k-plet} is a collection $(\PP^2,B_1), \ldots , (\PP^2,B_k)$ of plane curves, all of the same degree d, such that
\begin{enumerate}
\item all curves have the same combinatorial data ( for irreducible curves, this means the set of types of singular points), and
\item
the pairs $(\PP^2,B_1), \ldots , (\PP^2,B_k)$ are topologically non-equivalent.
%\item the curves are pairwise not equisingular deformation equivalent.
\end{enumerate}
If $k=2$ we speak of \emph{Zariski pair}
\end{defin}

There is natural way to produce many singular plane curves with the above mentioned singularities. Indeed, it is enough to consider branch curves of covering of $\PP^2$. To give such covering we will proceed in the following way: first we consider a surface of general type $S$ with ample canonical class, then we consider the natural immersion in a $\PP^n$ by a multicanonical system. Finally, we project the image generically  to $\PP^2$. This yields a covering $S \longrightarrow \PP^2$. To be more precise let us explain in details this procedure.  

\begin{defin}[cf. \cite{K99}]
\label{kulikov}
Let $B\subset \PP^2$ be an irreducible plane algebraic curve with ordinary cusps
and nodes as the only singularities. The curve $B$ is called \emph{generic branch curve} if there is a finite morphism $p:S\to \PP^2$ with $\deg p\geq3$ such that
\begin{enumerate}
\item
$S$ is a smooth irreducible projective surface,
\item
$p$ is unramified over $\PP\setminus B$,
\item
$p^*(B) = 2R + C$, where $R$ is a smooth irreducible reduced curve
and $C$ is a reduced curve,
\item
the morphism $p_{|R}:R\to B$ coincides with the normalization of $B$.
\end{enumerate}
In this case $p$ is called a \emph{generic covering of the projective plane}.
\end{defin}

%\blau{The definition says in fact explicitly that this is the case.( erase this part ? )}

%\begin{defin}[cf. \cite{Cat86}]
%A \emph{multiple plane} is a pair $(S,p)$ where $S$ is a connected smooth compact
%complex surface $S$ and $p$ is a finite holomorphic map $p:S\to \PP^2$.
%$(S,p)$ is called \emph{Chisini generic} if
%\begin{enumerate}
%\item
%the ramification divisor $R$ of $p$ is smooth and reduced,
%\item
%$p(R)=:B$ has only nodes and ordinary cusps as singularities,
%\item
%$p_{|R}:R\to B$ has degree $1$. ( In particular it is the normalization map.)
%\end{enumerate}
%In this case $B$ is called a \emph{generic branch curve}.
%\end{defin}

\begin{theo}\cite[Theorem 1.1]{CiF11} 
\label{cilibertoflamini}
Let $S \subset \PP^r$ be a smooth, irreducible, projective surface. Then the ramification curve on $S$ of a generic projection of $S$ to $\PP^2$ is smooth and irreducible and the branch curve in the plane is also irreducible and has only nodes and cusps, respectively, corresponding to two simple ramification points and one double ramification point.
\end{theo}

In particular the projection of $S$ onto $\PP^2$ is a generic covering in the sense of Def. \ref{kulikov}.

\begin{defin} A plane curve $B$ is called a \emph{generic m-canonical branch curve} if there exists a smooth surface $S$ with $mK_S$ very ample and a commutative diagram
\[
\begin{xy}
\xymatrix{%%
 S  \ar[rr]^{\phi_m} \ar[ddrr]_{p} & & \PP^{P_m-1}  \ar@{-->}[dd]%^{\lambda}
\\ 
\\
& & \PP^2
  }
\end{xy}
\]
where $\phi_m$ is the $m$-canonical map,  %$\lambda$ is a projection, 
$p$ is a generic covering and $B$ is the branch locus of $p$.
\end{defin}

%\begin{rem}
%Why do the additional properties hold?
%\begin{enumerate}
%\item
%if we work with Catanese definition, (3) is not obvious.
%By CiF the cusp points have only one preimage in R, so if they are not ramification
%points of the restriction $p_{|R}$ then $p$ is genericall bijective, hence of degree $1$.
%\item
%if we work with Kulikov definition, (3) and (4) are not obvious.
%\end{enumerate}
%\end{rem}

\begin{rem}
If $B'$ is an equisingular deformation of a generic $m$-canonical branch curve $B$. 
Then $B'$ is a generic branch curve.
Is it a generic $m$-canonical branch curve?

%{the genericity claim is possibly true by a paper of Michael Friedman} 
\end{rem}

In particular we want to study the corresponding m-canonical branch curve $(\PP^2,B)$ to the SIPs studied in \cite{LP16}. For such surfaces it holds the following lemma.

\begin{lem}
\label{lem_2K} 
Let $S$ be a regular surfaces isogenous to a product 
of holomorphic Euler number $x_{k,l}= 2^{l-3}(k^2+k-4)$
given by a $(\ZZ/2\ZZ)^k$ action with ramification structure of type $(2^{k(k+1)},2^{4+2^{l-k+1} })$. Then $K_S$ is ample and $2K_S$ is very ample for $k$ sufficiently large.
\end{lem}

\begin{proof}
First we note, that any smooth curve on $S$ has genus at least $3$.
Otherwise its preimage $\Lambda$ contains a component of genus at most $|G|+1=2^k+1$.
This component has to map surjectively onto at least one of $C_1$, $C_2$ in contradiction to the
fact that the genera of these curves are bigger than $2^k+1$.
In particular, there are no $-2$ curves on $S$, so $K_S$ is ample.

To proof the claim for the bicanonical map, first observe, that for $k$ sufficiently large 
it is either a birational morphism or it presents the standard case of non-birationality 
\cite{B73}.
Now the second alternative can be excluded, since any fibration of $S$ by curves of genus two implies
the existence of a smooth genus two curve on $S$ in contradiction to our fist observation.

Thus the bicanonical map is a birational morphism and even an embedding thanks to 
$K_S$ being ample.
\end{proof}

\begin{lem}
\label{conj}
Let $S_1$ be as before, than its fundamental group and its Euler number determine a unique irreducible
and connected component of the moduli space of surfaces of general type.
\end{lem}

\begin{proof}
By the weak rigidity theorem \ref{weak}, it suffices to show that $S_1$ and its complex conjugate
belong to one irreducible component.
To see this, notice we can apply our construction of $C_1,C_2$ with all branch points in the real locus
of $\PP^1$. Then the complex conjugate curves $\bar C_1,\bar C_2$ are constructed the same way,
after reverse the order in the system of generators and take the inverse of each generator 
\cite{BCG15}.
The last operation is trivial, since all generators are of order $2$, while the first can be undone with
the half twist in the braid group, since our group $G$ is abelian. Hence the SIP associated to
$C_1,C_2$ is isomorphic to its complex conjugate and therefore the claim follows. 
\end{proof}

Recall the weak positive  answers  to the Chisini conjecture by Kulikov. The proof of the conjecture was done in two steps, first he proved 
\begin{theo}[cf. \cite{K99}]\label{KulivovChisiniI} 
Let $B$ be the generic branch curve of a generic covering
$f:S\to \mathbb P^2$ of $\deg f = N$. If 
\begin{equation}
N > \frac{4(3d+g-1)}{2(3d+g-1)-c}. \label{in}
\end{equation}
Then $f$ is uniquely determined by $B$, thus Chisini's conjectures hold in this case.
\end{theo}

Then he proved the following theorem.
\begin{theo}[cf. \cite{K08}] 
\label{kulikov08}
Let $f\colon S \longrightarrow \PP^2$ be a generic projection. Then the generic covering $f$ 
is uniquely determined (up to isomorphism) by its branch
curve $B \subset \PP^2$ except in the case $S \cong \PP^2$ is embedded in $\PP^5$ by the 
Veronese embedding and $f$ is the restriction to $S$ of a linear projection 
from $\PP^5$ to $\PP^2$.
\end{theo}

We prove that the combinatorial datas of a generic m-canonical branch curve of a surface isogenous to a product $S$ depends only on the  group $G$ and a ramification structure for $S$.

\begin{lem}\label{lem_2caninv} Let $(\PP^2,B)$ be a generic 2-canonical branch curve of regular surfaces isogenous to a product $S$  then
\begin{equation}\label{eq_deg} d=\deg B=14c_1(S)^2,
\end{equation}
\begin{equation}\label{eq_cusps} c=68c_1(S)^2-c_2(S)
\end{equation}
\begin{equation}\label{eq_nodes} n=98(c_1(S)^2)^2-117c_1(S)^2+c_2(S)
\end{equation}
\begin{equation}\label{eq_eulernormalbranch} \chi_{top}(R)=-56c_1(S)^2
\end{equation}
\end{lem}

\begin{proof}
Let us denote by $\nu=4c_1(S)^2$ the degree  of $p\colon S \longrightarrow \PP^2$.
Let $R$ be the ramification locus of the generic covering $p$, which is smooth and the
normalization of $B$.
For a generic line $\ell\subset\PP^2$ we have the restriction $p|_\ell:\tilde\ell\to \ell$,
with $\tilde\ell$ in the linear system $2K_S$.

Since $p|_\ell$ is a simple covering of $\ell$ branched at $\ell\cap B$ Riemann Hurwitz gives us
\[
\chi_{top}(\tilde\ell) = \nu\chi_{top}(\ell) -\#\ell\cap B = 2 \nu -\deg B.
\]
With adjunction on $S$ we have $\chi_{top}(\tilde\ell)= -(2K_S)(2K_2+K_S)=-6c_1(S)^2$ and thus
obtain \eqref{eq_deg}.

The formula for ramification divisor of $p$ reads $K_S=-3\tilde\ell+R$ and adjunction on $S$
gives \eqref{eq_eulernormalbranch}:
\[
\chi_{top}(R) = -R(R+K_S) = -(3\tilde\ell+K_S)(3\tilde\ell+2K_S) = -7K_S(8K_S) = -56c_1(S)^2.
\]

To get an equation for the number of cusp $c$ we compute the topological Euler number of
$S$ from the map $p$ in comparison with the map $p|_R$. We note the following
\begin{enumerate}
\item
over $x\in\PP^2-B$ there are $\nu$ points of $S$ and no point of $R$,
\item
over $x\in B_{reg}$, the regular points of $B$ there are $\nu-1$ points of $S$ and one
point of $R$,
\item
over $x$ a node of $B$ there are $\nu-2$ points of $S$ and two
points of $R$,
\item
over $x$ a cusp of $B$ there are $\nu-2$ points of $S$ and only one point of $R$.
\end{enumerate}
Thus we can write the sum of the Euler numbers of $S$ and $R$ as follows
\[
c_2(S) + \chi_{top}(R) = \nu c_2(\PP^2) - c
\quad \implies \quad
c = 3\nu - \chi_{top}(R) - c_2(S) = 68 c_1(S)^2 - c_2(S).
\]
Finally we compute the Euler number along the normalization map to obtain
\[
\chi_{top}(R) = d(3-d) + 2n +2c
\quad \implies \quad
2n = \chi_{top}(R) +d^2 - 3d -2c
%= 196 c_1^2 -42 c_1 - 56 c_1 -136 c_1 - 2c_1
= 196 c_1^2 -234 c_1 - 2c_1
\]

\end{proof}

\begin{rem} The Lemma above can be easily generalized for any $m-$canonical curve.
\end{rem}

Now we can link SIPs with Zariski pairs.
\begin{theo}\label{theo_ZP} Let $S_1$ and $S_2$ be two surfaces isogenous to a product as in  Lemma \ref{lem_2K} with the same Euler number and $\pi_1(S_1) \neq \pi_1(S_2)$. Then for $m \geq 2$  the corresponding generic m-canonical branch curves $(\PP^2, B_1)$, $(\PP^2, B_2)$ are a Zariski pair.
\\
%\blau{
More precisely, they are distinguished by the fact that no isomorphism 
$\pi_1(\PP^2\setminus B_1)\to \pi_1(\PP^2\setminus B_2)$ exists, which maps the meridian class for $B_1$ to that for $B_2$.%}
\end{theo}

\begin{proof} 
%{\bf add Catanese ref} and 
We know that for $m\geq 2$ there exists pluricanonical embeddings $\iota_{m,i} \colon S_i \longrightarrow \PP^{P_m-1}$ for $i=1,2$. Then the generic linear projection of the pluricanonical image  to $\PP^2$ yields a generic  covering $p_i\colon S_i\longrightarrow \PP^2$. These are branched cover with branch $B_i$. 
By Lemma \ref{lem_2caninv} we see that the corresponding generic m-canonical branch curves 
$(\PP^2, B_1)$, $(\PP^2, B_2)$ have the same combinatorial data. 
%Indeed, these datas depends only on $\chi(S_i)$, hence only on the ramification datas of the surfaces $S_i$, that by hypothesis are the same. 

%Therefore, it is sufficient to prove that $\pi^K_1(B_1) \neq \pi^K_1(B_2)$.
Let us assume to the contrary that  $(\PP^2, B_1)$ $(\PP^2, B_2)$ are not a Zariski pair. This implies, by the definition of Zariski pairs, that 

\begin{enumerate}
\item
there exists an isomorphism %\blau{
\[ \varphi\colon\pi_1(\PP^2\setminus B_1) \longrightarrow \pi_1(\PP^2\setminus B_2),
\]%}
\item
which maps the conjugacy class of a meridian for $B_1$ to that of a meridian for $B_2$.
\end{enumerate}
%unique up to inner automorphisms, and even more that $\PP^2\setminus B_1$ is homeomorphic to $\PP^2\setminus B_2$. 
Under these assumptions we shall prove that $\pi_1(S_1) = \pi_1(S_2)$.

The topological unramified covers $S_i \setminus  p^{-1}_i(B_i) \longrightarrow \PP^2 \setminus B_i$ 
correspond to monodromy homomorphism onto the symmetric group $\sym_\nu$ up to
conjugation, where $\nu$ is the degree of $p_1,p_2$.
We choose two representatives, ie.\ group epimorphisms 
\[ \mu_i\colon \pi_1(\PP^2\setminus B_i, x_i) \longrightarrow \sym_\nu.
\]

The composition $\mu_2 \circ \varphi$ is an epimorphism  $\pi_1(\PP^2\setminus B_1, x_1) \longrightarrow \sym_{\nu}$,
which maps the class of meridians to the class of transpositions.
This monodromy morphism yields a topological covering of $\PP^2 \setminus B_1$ that by Riemann Existence Theorem
\cite{GrRe58} can be compactified  to a covering $S'_1 \longrightarrow \PP^2$
which is smooth due to the local analysis in \cite[\S1]{Cat86}.  
By the positive answer to the Chisini's conjecture by Kulikov, Thm.\ref{kulikov08}, there exists an isomorphism $f$ such that the following diagram commuts
 \begin{equation}\label{dia_chisini}
\xymatrix{
S_1 \ar[rr]^{f} \ar[dd]_{p_1} & & S'_1 \ar[dd]^{p'_1}  \\
\\
  \PP^2 \ar[rr]^{id}  & & \PP^2.
}
\end{equation}

%As mentioned above $\PP^2\setminus B_1$ is homeomorphic to $\PP^2\setminus B_2$. Threfore we have a topological covering $X_1\setminus p^{-1}_1(B_1) \longrightarrow \PP^2 \setminus B_2$, 
This yields the following  commutative diagram 

\void{
\begin{equation}\label{beta}
\xymatrix{
&  \pi_1(\PP^2\setminus B_1, x_1) \ar[rrr]^{\mu_2 \circ \varphi} \ar[ddd]_{\varphi} & & & \sym_{\nu} \ar[ddd]^{=}\\
\pi_1(X'_1\setminus (p'_1)^{-1}(B_1), f(\tilde{x_1})) \ar@{^{(}->}[ur]^{j'_1} \ar[rrr]^{\mu_2 \circ \varphi}  \ar@{-->}[d]_{\tilde{\varphi}} &  &  & \sym_{\nu-1} \ar[ur] \ar[d]^{=}\\
\pi_1(X_2\setminus p_2^{-1}(B_2), \tilde{x_2})  \ar@{^{(}->}[dr]_{j_2}  \ar[rrr]^{\mu_2} &  &  & \sym_{\nu-1} \ar[dr] \\
&  \pi_1(\PP^2\setminus B_2, x_2) \ar[rrr]^{\mu_2} &  & & \sym_{\nu} \\
}
\end{equation}
}

\begin{equation}\label{big}
\xymatrix{
\pi_1(S'_1\setminus (p'_1)^{-1}(B_1), f(\tilde{x_1})) \ar[rrrr]^{\mu_2 \circ \varphi} \ar@{_{(}->}[dr]_{j'_1}
\ar@{-->}[ddd]_{\tilde{\varphi}} &  
& 
& 
& \sym_{\nu-1} \ar[dl] \ar[ddd]^{=} \\
&  \pi_1(\PP^2\setminus B_1, x_1) \ar[rr]^{\mu_2 \circ \varphi} \ar[d]_{\varphi} 
& 
& \sym_\nu \ar[d]^{=}
& \\
& \pi_1(\PP^2\setminus B_2, x_2) \ar[rr]^{\mu_2} 
&  
& \sym_{\nu} \\
\pi_1(S_2\setminus p_2^{-1}(B_2), \tilde{x_2})  \ar@{^{(}->}[ur]^{j_2}  \ar[rrrr]^{\mu_2} 
&  
&  
&  
& \sym_{\nu-1} \ar[ul] \\
}
\end{equation}

We define $\tilde{\varphi}=\varphi \circ j'_1$ having image in  $\pi_1(\PP^2\setminus B_2, x_2)$. By commutativity $\mu_2 \circ \tilde{\varphi}=(\mu_2 \circ \varphi) \circ j'_1$ and has image in $\sym_{\nu-1}$. Therefore, $\tilde{\varphi}$ actually maps to $(\mu_2)^{-1}(\sym_{\nu-1})$ which is $\pi_1(X_2\setminus p^{-1}_1(B_2), \tilde{x_2}) $. 
Arguing in the same way with 
%\blau{
$\varphi^{-1}\circ j_2$ we see that $\tilde{\varphi}$ does in fact give the isomorphism with
the dashed arrow in \eqref{big}.

Let us write $(p'_1)^{-1}(B_1)=R_1+C_1$ and $p_2^{-1}(B_2)=R_2+C_2$, where $R_i$ are the ramification divisors. Let $D_i$ be a small transversal disc to $B_i$ for $i=1,2$, its pre-image is a union of discs $D_{R_i}$ and $D_{C_{i},1}, \ldots , D_{C_{i},\nu-2}$ transversal respectively to $R_i$ and $C_i$ for $i=1,2$. Let $m_{R_i}$, $m_{C_i, k}$ denote any corresponding meridians. We get the following diagram with exact rows.

\begin{equation}\label{alpha}
\xymatrix{
1 \ar[r] & \langle \langle  m_{R_1}, m_{C_1, k} \rangle \rangle_{k=1,\ldots \nu-2} \ar[r]  \ar@{-->}[dd]^{\tilde{\varphi}|} &  \pi_1(S'_1 \setminus (p'_1)^{-1}(B_1), f(\tilde{x_1})) \ar[r]  \ar[dd]^{\tilde{\varphi}}  & \pi_1(S'_1, f(\tilde{x_1}))  \ar[r]  \ar@{-->}[dd]^{\bar{\varphi}}  & 1 \\
\\ 
1 \ar[r] & \langle \langle  m_{R_2}, m_{C_2, k} \rangle \rangle_{k=1,\ldots \nu-2}  \ar[r]  &  \pi_1(S_2 \setminus p_2^{-1}(B_2), \tilde{x_2}) \ar[r]   & \pi_1(S_2, \tilde{x_1}) \ar[r] & 1
}
\end{equation}
We infer that this diagram yields a commutative diagram with the dashed arrows isomorphisms
by the following argument:\\
Notice that $(p'_1)_*(m_{R_1})$ is conjugate to the square of $m_{B_1}$, a meridian corresponding to the disc $D_1$. For the same reason the square of $m_{B_2}$ is conjugate to $(p_2)_*(m_{R_2})$. By assumption $\varphi(m_{B_1})$ is conjugate to $m_{B_2}$, a meridian corresponding to the disc $D_2$ in the group $\pi_1(\PP^2\setminus B_2, x_2)$. 
Let $n_1$ be the element conjugate to $m_{B_2}$ such that $\varphi((p'_1)_*(m_{R_1}))=n^2_1$ and $n_2$ be the element conjugate to $m_{B_2}$ such that $(p_2)_*(m_{R_2})=n^2_2$. 
%Therefore $\varphi((p'_1)_*(m_{R_1}))=:n_1$ is conjugate to $(p_2)_*(m_{R_2})=:n_2$ in $\pi_1(\PP^2\setminus B_2, x_2)$. 
We want to prove that $\tilde{\varphi}(m_{R_1})$ is conjugate to $m_{R_2}$ in  $\pi_1(S_2 \setminus p_2^{-1}(B_2), \tilde{x_2})$. We have the following diagram of groups and morphisms.
\[
\xymatrix{
 \langle \langle  m_{R_i}, m_{C_i, k} \rangle \rangle_{k=1,\ldots \nu-2} \ar[r] &  \pi_1(S'_1 \setminus (p'_1)^{-1}(B_1), f(\tilde{x_1})) \ar[r]^{\qquad\: p_1'}  \ar[dd]^{\tilde{\varphi}} &  \pi_1(\PP^2\setminus B_1, x_1) \ar[r]   \ar[dd]^{\varphi}&\sym_{\nu}/\sym_{\nu-1}\\ \\
\langle \langle  m_{R_i}, m_{C_i, k} \rangle \rangle_{k=1,\ldots \nu-2} \ar[r] &  \pi_1(S_2 \setminus p_2^{-1}(B_2), \tilde{x_2}) \ar[r]^{\quad p_2}  &  \pi_1(\PP^2\setminus B_2, x_2) \ar[r]^{\quad\eta} & \sym_{\nu}/\sym_{\nu-1}
}
\]
For $i=1,2$, the image of $n_i$ in $\sym_{\nu}/\sym_{\nu-1}$ is non trivial, hence it is a transposition involving $\nu$. We can suppose, up to conjugation in  the pre-image of $\sym_{\nu-1}$, that this transposition is $(\nu-1,\nu)$ for both $n_1$ and $n_2$.  Let $\alpha$ a conjugation element such that $n_1^{\alpha}=n_2$, then $\eta(\alpha) \in \langle \sym_{\nu-2},(\nu-1,\nu)\rangle$, if we consider $n_1 \alpha$ instead of $\alpha$ we see that we can choose $\alpha$ such that $\eta(\alpha) \in \sym_{\nu-2}$. 

This means that  $\tilde{\varphi}(m_{R_1})$ is conjugate to $m_{R_2}$ in  $\pi_1(S_2 \setminus p_2^{-1}(B_2), \tilde{x_2})$. Analogously one can prove that  $\tilde{\varphi}(m_{C_1,k})$ is conjugate to $m_{C_2,k'}$ in  $\pi_1(S_2 \setminus p_2^{-1}(B_2), \tilde{x_2})$. 
Thus we get the map $\tilde{\varphi}|$, and arguing the same way with $\tilde{\varphi}^{-1}$ we see that  $\tilde{\varphi}|$ is an isomorphism.
This proves that $\bar{\varphi}$ is a well defined isomorphism. But then we arrive at a contradiction to the hypothesis that $\pi_1(S_1) \neq \pi_1(S_2)$. Therefore,  $(\PP^2, B_1)$, $(\PP^2, B_2)$ are  a Zariski pair. 
\end{proof}

\begin{rem} The Theorem \ref{theo_ZP} can be generalized to any pair surfaces isogenous to a product with the same Euler number and $\pi_1(S_1) \neq \pi_1(S_2)$ provided  $mK_S$ is very ample.
\end{rem}

%%%%%%%%%%%%%%%%%%%%%%%%%%%%%%%%%%%%%%%%%%%%%%%%%%

\section{Zariski multiplets}

In this section we study Zariski multiples arising as branch curves of coverings of $\PP^2$ by SIPs. Moreover,  the knowledge of the number of connected components of the moduli space of SIPs with given invariants will enable us to produce Zariski  multiples whose number grows sub-exponentially with respekt ot the degree.  

%\begin{defin} Let $B(d,c,n)$ be the set of curves of degree d with c cusps and n nodes,
%which are branch curves of a Chisini general projection.

%Let $B_{2}(d,c,n) \subset B(d,c,n)$ be the set of generic 2-canonical branch curves of degree d with c cusps and n nodes.

%Let $B_{I,2}(d,c,n) \subset B(d,c,n)$ be the subset of generic 2-canonical branch curves associated to
%regular surfaces isogenous to a product.
%\end{defin}

%\begin{rem}
%They are both spaces...
%\end{rem}

Recall that (see \cite{Gie77}) once two positive integers $x,y$ are fixed there exists a quasi-projective coarse moduli space $\mathcal{M}_{y,x}$ of canonical models of surfaces of general type with 
{$x=\chi(S)=\chi(\mathcal{O}_S)$} and $y=K^2_S$. The number $\iota(x,y)$, resp.\ $\gamma(x,y)$, of irreducible, resp.\ connected, components of  $\mathcal{M}_{y,x}$ is bounded from above by a function of $y$ and we shall also consider the number $\iota^0(x,y)$ of components containing regular surfaces, i.e., \ $q(S)=0$. In \cite{LP16} we proved the following fact. 

\begin{theo}[cf. Theorem 1.1 \cite{LP16}]\label{theo_SIPs}\sloppy Let $h=h_{k,l}$ be number of connected components of the moduli space of surfaces of general type which contain regular surfaces isogenous to a product 
of holomorphic Euler number $x_{k,l}= 2^{l-3}(k^2+k-4)$
given by a $(\ZZ/2\ZZ)^k$ action with ramification structure of type $(2^{k(k+1)},2^{4+2^{l-k+1} })$.

If $k,l$ are positive integers with $l>2k$, then 
\begin{equation}
\label{eq_compSIPs}
h\quad \geq \quad 2_{\displaystyle\phantom{A}}^{\sqrt[\uproot{2}2+\varepsilon]{x_k}}\qquad \text{for} \quad k \rightarrow \infty,
\end{equation}
where $\varepsilon$ is the positive real number such that $l=(2+\varepsilon)k$.
%In particular, given any sequence $\alpha_i$ which is positive, increasing and bounded by $\frac12$ from above,
%we obtain increasing sequences $x_i$ and $y_i=8x_i$ with
%\begin{equation}\label{eq_compSIPs}
%\iota^0(x_i,y_i) \quad\geq\quad y_i^{(y^{\alpha_i}_i)} .  
%\end{equation}
\end{theo}

Using the theorem above it is not hard to prove the main theorem.

\begin{theo}\label{thm main2}
There is a $N_d$-multiplet of irreducible plane curves of degree $d$
with 
\[
n_d= \frac12d^2 - \frac{233}{28}d \mbox{ nodes},\quad
c_d= \frac{135}{28} d \mbox{ cusps}
\]
for any $\varepsilon>0$ and infinitely many $d$ such that
\[
\log_2 N_d \quad \geq\quad (\mbox{$\frac8{14}$}d)^{\frac{1}{2+\varepsilon}}. 
%{\scriptstyle{\frac8{14}}}
\]
\end{theo}

\begin{proof} We have seen that to each SIP with holomorphic Euler number $x_{k,l}= 2^{l-3}(k^2+k-4)$
given by a $(\ZZ/2\ZZ)^k$ action with ramification structure of type $(2^{k(k+1)},2^{4+2^{l-k+1} })$ there corresponds a deformation class of $2$-canonical branch curves. 

Moreover the number $h_{k,l}$ of connected components of moduli 
of these surfaces is in bijection with the isomorphism
classes of their fundamental groups, by weak rigidity \ref{weak} and Lemma \ref{conj}.
Therefore by Theorems \ref{theo_ZP} we get a Zariski multiplet of cardinality $h_{k,l}$.
By Theorem \ref{theo_SIPs}, in particular by \eqref{eq_compSIPs}, 
for given $\varepsilon>0$ and 
we can estimate this number in terms
of $8x=c_1(S)^2$. 
Finally we can replace $x$ by the degree $d$ of the branch curve, since $d= 14c_1^2(S)$, see Lemma \ref{lem_2caninv}. 
\end{proof}

As a final remark, we want to mention that Catanese \cite{Cat92} proves that 
the number $\iota^0(y,x)$ of components containing regular surfaces, ie.\ $q(S)=0$, has
an exponential upper bound in $K^2$, more precisely \cite[p.592]{Cat92} gives the following inequality
\[
\iota^0(x,y) \leq y^{77y^2}. 
\]
Though this can be used to give an upper bound on equisingular deformation classes of
generic $2$-canonical branch curves, there may be large Zariski multiplets containing
irreducible curves, which are not generic $2$-canonical.

%%%%%%%%%%%%%%%%%%%%%%%%%%%%%%%%%%%%%%%%%%%%%%%%%%
%%%%%%%%%%%%%%%%%%%%%%%%%%%%%%%%%%%%%%%%%%%%%%%%%%

\end{document}